\documentclass[12pt,bezier]{article}
\usepackage{amsmath}
\usepackage{amsfonts,amsthm,amssymb}
\usepackage{amsfonts}
\usepackage{graphics}
\usepackage{ifpdf}
\usepackage{color}
\newtheorem{lem}{Lemma}
\newtheorem{thm}[lem]{Theorem}
\newtheorem{cor}[lem]{Corollary}
\newtheorem{obs}[lem]{Observation}
\newtheorem{rem}[lem]{Remark}

\newtheorem{con}[lem]{Conjecture}

\begin{document}
\title{Anti-$k$-labeling of graphs}

\author{Xiaxia Guan$^a$, Shurong Zhang$^a$, Rong-hua Li$^b$, Lin Chen$^c$, Weihua Yang$^a$\footnote{Corresponding author.
 E-mail addresses: ywh222@163.com(W.Yang).} \\
  \small $^a$Department of Mathematics, Taiyuan University of
Technology, Shanxi Taiyuan-030024, China \\
  \small $^b$School of Computer Science \& Technology, Beijing Institute of Technology, Beijing-100081, China\\
  \small $^c$Lab. de Recherche en Informatique (LRI), CNRS, University
of Paris-Sud XI, 91405 Orsay, France
 }

%\date{}

\maketitle

{\small{\bf Abstract.}\quad
It is well known that the labeling
problems of graphs arise in many (but not limited to) networking and
telecommunication contexts. In this paper we introduce the anti-$k$-labeling problem of graphs which we seek to minimize the similarity (or distance) of neighboring nodes. For example, in the fundamental frequency assignment problem in wireless networks where each node is
assigned a frequency, it is usually desirable to limit or minimize the
frequency gap between neighboring nodes so as to limit interference.

Let $k\geq1$ be an integer and $\psi$ is a labeling function (anti-$k$-labeling) from $V(G)$ to $\{1,2,\cdots,k\}$ for a graph $G$. A  {\em no-hole anti-$k$-labeling} is an anti-$k$-labeling using all labels between 1 and $k$. We define $w_{\psi}(e)=|\psi(u)-\psi(v)|$ for an edge $e=uv$ and $w_{\psi}(G)=\min\{w_{\psi}(e):e\in E(G)\}$ for  an anti-$k$-labeling $\psi$  of the graph $G$.  {\em The anti-$k$-labeling number}  of a graph $G$, $mc_k(G)$ is $\max\{w_{\psi}(G): \psi\}$. In this paper, we first show that $mc_k(G)=\lfloor \frac{k-1}{\chi-1}\rfloor$, and the problem that determines $mc_k(G)$  of graphs is NP-hard. We mainly obtain the lower bounds on no-hole anti-$n$-labeling number  for trees, grids and $n$-cubes.

\section{Problems}

All graphs considered here are simple and finite. Definitions which are not given here may be found in \cite{bondy}. Let $k\geq1$ be an integer. An {\em anti-$k$-labeling} $\psi$ of a graph $G$ is a mapping from $V(G)$ to $\{1,2,\cdots,k\}$. An anti-$k$-labeling $\psi$ of $G$ is called a {\em no-hole anti-$k$-labeling} if it uses all labels between 1 and $k$. We define $w_{\psi}(e)=|\psi(u)-\psi(v)|$ ($w^{nh}_{\psi}(e)=|\psi(u)-\psi(v)|$) for an edge $e=uv$ and $w_{\psi}(G)=\min\{w_{\psi}(e):e\in E(G)\}$ ($w^{nh}_{\psi}(G)=\min\{w^{nh}_{\psi}(e):e\in E(G)\}$) for  an anti-$k$-labeling $\psi$ (a no-hole anti-$k$-labeling $\psi$) of the graph $G$. The {\em anti-$k$-labeling number} ({\em the no-hole anti-$k$-labeling number}) of a graph $G$, $mc_k(G)$ ($mc^{nh}_{k}(G)$), is $\max\{w_{\psi}(G): \psi\}$ ($\max\{w^{nh}_{\psi}(G): \psi\}$). We refer to a labeling $\psi$ with $w_{\psi}(G)=mc_k(G)$ ($w^{nh}_{\psi}(G)=mc^{nh}_{k}(G)$) is called an {\em optimal anti-$k$-labeling} (an {\em optimal no-hole anti-$k$-labeling}) for a graph $G$.   Such  (no-hole) anti-$k$-labeling number problem is our focus in this paper.

The above labeling problem represents a generic class of labeling
problems arising in many (but not limited to) networking and
telecommunication contexts, in which we seek to minimize the similarity
(or distance) of neighboring nodes. For example, in the fundamental
frequency assignment problem in wireless networks where each node is
assigned a frequency, it is usually desirable to limit or minimize the
frequency gap between neighboring nodes so as to limit interference.
Another example relates to the content sharing systems such as
peer-to-peer file sharing systems, where resources (e.g., files) are
replicated at network nodes to reduce resource retrieval time and
increase system robustness. In these systems, to maximize performance
gain, we usually want to place different items in the vicinity of each
node or to place the same items far from each other.

These problems can be cast to the labeling problem where we seek a node
labeling maximizing the minimum labeling distance among neighboring nodes.
Surprisingly, this labeling problem has not yet been analyzed (not even
formulated in a mathematical sense).

In some sense, our focus problem is also  a generalization of  vertex-coloring problem of graphs: Find the  smallest number $m$ such that $G$ has a labeling $f$: $V(G)\rightarrow[0,m)$ satisfying $|f(x)- f(Y)| \geq h$ for $xy \in E(G)$. In fact, $mc_k(G) > 0$ if and only if $k\geq \chi(G)$ for a graph $G$, where $ \chi(G)$ is the chromatic number of  graph $G$. Hence, $ \chi(G)$ is the minimum number such that $mc_k(G) > 0$ for a graph $G$. Since  determining the chromatic number of graphs is NP-hard, the anti-$k$-labeling problem is also NP-hard.

Another related labeling problem (namely, $L(2,1)$-labeling) will be mentioned in Section 4.

\section{$mc_k(G)$ and $\chi(G)$ of graphs}

\begin{obs}\label{obs 1}
If  $H$ is a subgraph of $G$, then  $mc_k(H)\geq mc_k(G)$.
\end{obs}
\begin{proof}
 Clearly, for an arbitrary anti-$k$-labeling $\psi$, $w_{\psi}(H)\geq w_{\psi}(G)$ holds. Suppose $\psi$ is an optimal anti-$k$-labeling of $G$ (i.e., $w_{\psi}(G)=mc_{k}(G)$), then  $w_{\psi}(H)\geq w_{\psi}(G)=mc_k(G)$. Hence, $mc_k(H)\geq mc_k(G)$ by the definition of anti-$k$-labeling number.
\end{proof}

Suppose $G_{1}$ and $G_{2}$ are two graphs with $V(G_{1})\cap V(G_{2})=\emptyset$. The {\em union} $G$ of $G_{1}$ and $G_{2}$, denoted by $G=G_{1}\cup G_{2}$, is the graph whose vertex set is $V(G_{1})\cup V(G_{2})$, and edge set is $E(G_{1})\cup E(G_{2})$.

\begin{obs}\label{obs 3}
If $G=G_{1}\cup G_{2}$, then $mc_k(G)=\min\{mc_k(G_{1}),mc_k(G_{2})\}$.
\end{obs}
\begin{proof}
$mc_k(G)\leq \min\{mc_k(G_{1}),mc_k(G_{2})\}$ following from Observation \ref{obs 1} and $G_{1}$ and $G_{2}$ are subgraphs of $G_{1}\cup G_{2}$. On the other hand, an anti-$k$-labeling of $G_{1}$ together with an anti-$k$-labeling of $G_{2}$ makes an anti-$k$-labeling $\psi$ of $G_{1}\cup G_{2}$ so that $\omega_\psi(G)\geq \min\{mc_k(G_{1}),mc_k(G_{2})\}$. Hence $mc_k(G)\geq \min\{mc_k(G_{1}),mc_k(G_{2})\}$.
\end{proof}

\begin{thm}\label{thm 4}
Let $G$ be  a  graph with chromatic number $\chi=\chi(G)\geq 2$, then $mc_k(G)=\lfloor \frac{k-1}{\chi-1}\rfloor$  for all $k$.
\end{thm}
\begin{proof}
 We first show that $mc_k(G)\geq\lfloor \frac{k-1}{\chi-1}\rfloor$. It suffices to show that there exists an anti-$k$-labeling  $\psi$ such that $w_\psi(G)=\lfloor \frac{k-1}{\chi-1}\rfloor$ for a graph $G$ with chromatic number $\chi=\chi(G)$. Let $V_{1},V_{2},\ldots,V_{\chi}$ be a proper $\chi$-coloring of $G$. Then we consider the following labeling $\psi$: label the vertices of $V_{1}$ by 1,  $\ldots$ , label the vertices of $V_{i}$ by $1+(i-1)\lfloor \frac{k-1}{\chi-1}\rfloor$, $\ldots$ , label the vertices of $V_{\chi}$ by $1+(\chi-1)\lfloor \frac{k-1}{\chi-1}\rfloor (\leq k)$. Since $w_{\psi}(G)=\min\{w_{\psi}(e):e\in E(G)\}=\lfloor \frac{k-1}{\chi-1}\rfloor$ due to $V_{i}$ ($i=1,2,\ldots,\chi$) being an independent set. Hence, $mc_k(G)\geq\lfloor \frac{k-1}{\chi-1}\rfloor$.

We next show that $mc_k(G)\leq\lfloor \frac{k-1}{\chi-1}\rfloor$. Let $\psi$ be an optimal anti-$k$-labeling of $G$ and $(V_{1},V_{2},\ldots,V_{k})$ be a partition  of   $V(G)$ under $\psi$, where the vertices in $V_{i}$ have label $i$, $i=1,2,\ldots,k$. Assume $mc_k(G)\geq\lfloor\frac{k-1}{\chi-1}\rfloor+1$.   We colour the vertices of $V_{1},V_{2},\ldots,V_{\lfloor\frac{k-1}{\chi-1}\rfloor+1}$ with color $c_1$, $\cdots$, colour the vertices of $V_{(i-1)\lfloor\frac{k-1}{\chi-1}\rfloor+i},V_{(i-1)\lfloor\frac{k-1}{\chi-1}\rfloor+i+1},\ldots,V_{i\lfloor\frac{k-1}{\chi-1}\rfloor+i}$ with color $c_i$ ($i=1,2,\ldots,\chi-2$), and colour the vertices of $V_{(\chi-2)\lfloor\frac{k-1}{\chi-1}\rfloor+\chi-1},V_{(\chi-2)\lfloor\frac{k-1}{(\chi-1)}\rfloor+\chi},\ldots,V_{k}$ with color $c_{\chi-1}$. Note $k\leq(\chi-1)(\lfloor\frac{k-1}{\chi-1}\rfloor)+\chi-1$. And the vertices of $V_{i}$ are not adjacent to the vertices of $V_{j}$ ($1 \leq j\leq k$), $j=i-\lfloor\frac{k-1}{\chi-1}\rfloor,i-\lfloor\frac{k-1}{\chi-1}\rfloor+1,\ldots,i+\lfloor\frac{k-1}{\chi-1}\rfloor$ by the assumption $mc_k(G)\geq\lfloor\frac{k-1}{\chi-1}\rfloor+1$. Thus, the vertices of coloring $c_i$ ($i=1,2,\ldots,\chi-1$) are not adjacent. This implies a proper $(\chi-1)$-coloring of  $G$,  a contradiction. Therefore $mc_k(G)=\lfloor \frac{k-1}{\chi-1}\rfloor$.
\end{proof}

By Theorem \ref{thm 4},   $mc_{k'}(G)\geq mc_k(G)$ holds for $k'\geq k$.

\section{$mc^{nh}_n(G)$ of graphs}

In this section we consider no-hole anti-$k$-labeling for $k=n$.

\begin{obs}\label{obs 5}
If $G^{\prime}$ is a spanning graph of $G$, then $mc^{nh}_n(G^{\prime})\geq mc^{nh}_n(G)$.
\end{obs}
\begin{proof}
Suppose $mc^{nh}_n(G)=l$ with an optimal labeling $\psi$.  Let $w^{nh}_{\psi}(G^{\prime})=w^{nh}_{\psi}(e)$. Then $mc^{nh}_n(G^{\prime})\geq w^{nh}_{\psi}(e)\geq w^{nh}_{\psi}(G)=l$ by the definitions. Therefore, $mc^{nh}_n(G^{\prime})\geq l$.
\end{proof}

\begin{obs}\label{obs 6}
For a  graph $G$ with $n$ vertices, $mc_n(G)\geq mc^{nh}_n(G)$ holds for all $n\geq2$.
\end{obs}
\begin{proof}
It is obvious that $mc_n(G)\geq mc^{nh}_n(G)$.
\end{proof}

We denote by $\delta$ and $\Delta$ the minimum degree and maximum degree of a graph $G$. We have the following.
\begin{obs}\label{obs 7}
For a  connected graph $G$ with $n$ vertices,  $mc^{nh}_n(G)\geq 1$ and $mc^{nh}_n(G)\leq\min\{n-\Delta,\lfloor\frac{n-1}{\chi-1}\rfloor,\lfloor\frac{n-\delta+1}{2}\rfloor\}$ hold for all $n\geq2$.
\end{obs}
\begin{proof}
For  each no-hole anti-$n$-labeling $\psi$, $w^{nh}_{\psi}(G)\geq 1$. Thus, $mc^{nh}_n(G)\geq 1$.

Note that the vertex with the maximum degree has $\Delta$ neighbors which have distinct labels for any no-hole anti-$n$-labeling. Then $mc^{nh}_n(G)\leq n-\Delta$.

Let  $v$ be the vertex having label $\lceil \frac{n}{2}\rceil$ for a no-hole anti-$n$-labeling $\psi$ of $G$, then there is an edge $e$ incident to $v$ so that $w^{nh}_{\psi}(e)\leq \lfloor\frac{n-\delta+1}{2}\rfloor$ since there is at least $\delta$ vertices adjacent to $v$ in $G$. Therefore $mc^{nh}_n(G)\leq \lfloor\frac{n-\delta+1}{2}\rfloor$.

It is clear that  $mc^{nh}_n(G)\leq mc_n(G)\leq \lfloor\frac{n-1}{\chi-1}\rfloor$ by Observation \ref{obs 6} and Theorem \ref{thm 4}. Thus, the claim holds.
\end{proof}

Let $G$ be a simple graph. The complement graph $G^{c}$ of $G$ is the simple graph with vertex set $V(G)$, two vertices being adjacent in  $G^{c}$ if and only if they are not adjacent in $G$.

\begin{obs}\label{obs 8}
For a simple graph $G$,   $mc^{nh}_n(G)\geq2$ if and only if there exists a Hamilton path for the complement graph $G^{c}$.
\end{obs}

\begin{proof}
Suppose that $mc^{nh}_n(G)\geq2$ with an optimal no-hole anti-$n$-labeling $\psi$. We may assume without loss of generality that the vertex $v_{i}$ is  labeled  $i$ for the labeling $\psi$. Then $v_{i}v_{i+1}$ ($i=1,2,\ldots,n-1$) does not belong to $E(G)$, that is, $v_{i}v_{i+1}$ ($i=1,2,\ldots,n-1$)  belongs to $E(G^{c})$. Hence the path $v_{1}\rightarrow v_{2}\rightarrow \ldots \rightarrow v_{n}$ is a Hamilton path of $G^{c}$.

Conversely, suppose that $v_{1}\rightarrow v_{2}\rightarrow \ldots \rightarrow v_{n}$ is a Hamilton path of $G^{c}$. We label the vertex $v_{i}$ by $i$. Note that $v_{i}v_{i+1}$ ($i=1,2,\ldots,n-1$) does not belong to $E(G)$. Then $w^{nh}_{\psi}(e)\geq2$ for any edge $e$. Hence $mc^{nh}_n(G)\geq2$.
\end{proof}

By Observation \ref{obs 8}, one can see that the no-hole anti-$n$-labeling number implies some structural properties of graphs.

\subsection{$mc^{nh}_n(G)$ of complete multipartite graphs}

\begin{thm}\label{thm 9}
Let $K_{n_1,\cdots,n_t}$ be a complete multipartite graph with $n$ vertices,  then $mc^{nh}_n(K_{n_1,\cdots,n_t})=1$ holds for all $n\geq2$, where $n_1
+\cdots+n_t=n$.
\end{thm}
\begin{proof}
 It is clear that $mc^{nh}_n(K_{n_1,\cdots,n_t})\geq 1$ according to Observation \ref{obs 7}.

 Now we prove $mc^{nh}_n(K_{n_1,\cdots,n_t})\leq 1$. Note that $K^{c}_{n_1,\cdots,n_t}=K_{n_1}\cup K_{n_2}\cup \cdots\cup K_{n_t}$ consisting of $t$ disjoint complete graphs. Thus,  there is no    Hamilton paths in $K^{c}_{n_1,\cdots,n_t}$. Then we have $mc^{nh}_n(K_{n_1,\cdots,n_t})<2$ by Observation \ref{obs 8}.
 \end{proof}

We next consider an example for graph operations. Suppose $G_{1}$ and $G_{2}$ are two graphs with disjoint vertex sets. The {\em join} $G$ of $G_{1}$ and $G_{2}$, denoted by $G=G_{1}+G_{2}$, is the graph obtained from $G_{1}\cup G_{2}$  by adding all edges between vertices in  $V(G_{1})$ and vertices in  $V(G_{2})$.

\begin{cor}\label{cor 11}
If $G=G_{1}+G_{2}$, then $mc^{nh}_n(G)=1$.
\end{cor}
\begin{proof}
If $G=G_{1}+G_{2}$, then $G^{\prime}$ is a spanning subgraph of $G$, where  $G^{\prime}$ is a complete bipartite graph with bipartition ($V(G_{1}),V(G_{2})$). Hence $mc^{nh}_n(G)=1$ by Observation~\ref{obs 5} and Theorem~\ref{thm 9}.
\end{proof}

\subsection{$mc^{nh}_n(G)$ of trees}

\begin{thm}\label{thm 12}
Let $P_n$ be a path  on $n$ vertices. Then $mc^{nh}_n(P_n)=\lfloor\frac{n}2\rfloor$.
\end{thm}

\begin{proof}
Since $\delta(P_n)=1$, $mc^{nh}_n(P_n)\leq \lfloor\frac{n}2\rfloor$ according to Observation ~\ref{obs 7}.

Let $v_{1},v_{2},\ldots, v_{n}$ be vertices of $P_{n}$ such that $v_{i}$ is adjacent to $v_{i+1}$, $1\leq i\leq n-1$. Now we show that $mc^{nh}_n(P_n)\geq \lfloor\frac{n}2\rfloor$. It suffices to show that there is a no-hole anti-$n$-labeling $\psi$ such that $w^{nh}_{\psi}(P_n)=\lfloor\frac{n}2\rfloor$ for $P_{n}$. Consider the following labeling:

(i) If $n$ is even, then we define
\[\psi(v_{i})=\left\{\begin{array}{ll}
\frac{n}2-\frac{i-1}2&\text{$i$ is odd},\\
n+1-\frac{i}2&\text{$i$ is even}.
\end{array}\right.\]

(ii) If $n$ is odd, then we define
\[\psi(v_{i})=\left\{\begin{array}{ll}
\frac{n+1}2-\frac{i-1}2&\text{$i$ is odd},\\
n+1-\frac{i}2&\text{$i$ is even}.
\end{array}\right.\]

\input{g3.TpX}

Clearly, for each $e\in E(P_{n})$, $w_{\psi}(e)$ is $\frac{n}2$ or $\frac{n}2+1$ for even $n$, and $\frac{n-1}2$ or $\frac{n+1}2$ for odd $n$. Hence $mc^{nh}_n(P_n)\geq \lfloor\frac{n}2\rfloor$.
\end{proof}

We note that a tree is a bipartite graph. A $leaf$ in a tree is a vertex of degree 1.

\begin{lem}\label{lem 13}
For a tree $T_n$ with bipartition $(X_{1},X_{2})$,  and  $|X_{1}|< |X_{2}|$,  then $X_2$ contains a leaf of  $T_n$.
\end{lem}

\begin{proof}
By contradiction, suppose that $X_2$ contains no leaves of $T_n$.   Let $Y_{0}=\{u:d(u)=1,u\in V(T_n)\}$, $Y_{i}=\{v:\exists u\in Y_{0}$ so that $d(u,v)=i\}\backslash \cup ^{i-1}_{j=0} Y_{j}$, $m=\max\{i:Y_{i}\neq\emptyset\}$. Note $Y_{i}\in X_{1}$ ($Y_{i}\in X_{2}$, resp.) for even (odd, resp.) $i\leq m$.  And $Y_{i}$ is an independent set for all $i\leq m$. Thus, $|Y_{i+1}|\leq|Y_{i}|$ ($i=0,1,\ldots,m-1$) due to $T_n$ without cycle.

If $m$ is even, then $|X_{2}|=|Y_{1}|+|Y_{3}|+\ldots +|Y_{m-1}|\leq|Y_{0}|+|Y_{2}|+\ldots +|Y_{m-2}|<|Y_{0}|+|Y_{2}|+\ldots +|Y_{m-2}|+|Y_{m}|=|X_{1}|$,  a contradiction. If $m$ is odd, then $|X_{2}|=|Y_{1}|+|Y_{3}|+\ldots +|Y_{m}|\leq|Y_{0}|+|Y_{2}|+\ldots +|Y_{m-1}|=|X_{1}|$, a contradiction. Thus, $X_2$ contains  a leaf of  $T_n$.
\end{proof}

\begin{thm}\label{thm 14}
For a tree $T_n$ with bipartition $(X_{1},X_{2})$, $|X_{i}|=q_{i}$, $i=1,2$, we have $mc^{nh}_n(T_n)\geq q= min\{q_{1},q_{2}\}$.

\end{thm}

\begin{proof}
Without loss of generality, we suppose $q_{1}\leq q_{2}$, i.e.,  $q=q_{1}$. We   show that  $mc^{nh}_n(T_n)\geq q$ by giving a no-hole anti-$n$-labeling $\psi$ with $w^{nh}_\psi(T_n)\geq q$. The result clearly holds for  $n=1,2$. If $n=3$, then $T_3=P_3$.  Let  $T_3=P_3=v_{1}v_{2}v_{3}$. Then $v_{2}\in X_{1}$ and $v_{1}, v_{3}\in X_{2}$.  We have $\psi(v_{1})=3$, $\psi(v_{2})=1$, and $\psi(v_{3})=2$, and $w^{nh}_\psi(T_3)\geq 1=q$ according to Theorem \ref{thm 12}. Hence, $mc^{nh}_n(T_n)\geq q$ for $n=3$.  Moreover,  each vertex of $X_{1}$ ($X_{2}$, resp.) is labeled by $i\leq q$ ($i>q$, resp.) in the labeling $\psi$.

We next construct a  no-hole anti-$n$-labeling $\psi$ of $T_n$ by induction on $n\geq 4$ such that each vertex of $X_{1}$ ($X_{2}$, resp.) is labeled by $i\leq q$ ($i>q$, resp.). We assume that  $\psi^{\prime}$ is a no-hole anti-$k$-labeling of $T_k$ satisfying the requirement for $k<n$. We label $T_n$ based on the labeling $\psi^{\prime}$ of $T_k$ as below.

 {\bf Case 1.}  $q_{1}<q_{2}$.

By Lemma \ref{lem 13},  there exists a  leaf $u\in X_{2}$  of $T_n$.   Let $T_{n-1}=T_n-u$.  Clearly, $|X_{1}(T_{n-1})|=|X_{1}(T_{n})|=q_{1}=q$, $|X_{2}(T_{n-1})|=|X_{2}(T_{n})|-1=q_{2}-1$ and $q_{1}\leq q_{2}-1$. By the induction hypothesis, there exists a  no-hole anti-$(n-1)$-labeling  $\psi^{\prime}$  so that $w^{nh}_{\psi^{\prime}}(T_{n-1})\geq q$. We obtain the labeling $\psi$ by labeling the vertex $u$ by $n$. It is obvious that $w^{nh}_\psi(T_n)\geq q$, and  each vertex of $X_{1}$ ($X_{2}$, resp.) is labeled by $i\leq q$ ($i>q$, resp.) in the labeling $\psi$ (see Fig.2(2)).

 {\bf Case 2.}  $q_{1}=q_{2}=q=\frac{n}{2}$.

Clearly, there is a vertex (say $u$) whose neighbors are all leaves except one vertex. Without loss of generality, we assume $u\in V(X_{2})$ and $u$ has $m$ leaves as its neighbors.  We consider the graph $T_{n-m-1}$  obtained from $T_{n}$ by removing the vertex $u$ and the $m$ neighbors (the $m$ leaves) of $u$. Note $|X_{1}(T_{n-m-1})|=|X_{1}(T_{n})|-m=\frac{n}{2}-m$,
$|X_{2}(T_{n-m-1})|=|X_{2}(T_{n})|-1=\frac{n}{2}-1$.
 By the induction hypothesis, there exists a  no-hole anti-$(n-m-1)$-labeling $\psi^{\prime}$  so that $w^{nh}_{\psi^{\prime}}(T_{n-m-1})\geq \frac{n}{2}-m$.

We now label  $T_{n}$ by the following rules (i.e., $\psi$): relabel  the vertex with label   $i>\frac{n}{2}-m$  in $T_{n-m-1}$ by $i+m$, label the vertex $u$  by $n$, and  label the $m$ neighbors of $u$ by $\frac{n}{2}-m+1,\frac{n}{2}-m+2,\cdots,\frac{n}{2}$. Clearly,  $\psi(v)\leq \frac{n}{2}$ ($\psi(v)>\frac{n}{2}$, resp.) for $v\in X_{1}$ ($v\in X_{2}$, resp.) in the  labeling $\psi$ of $T_n$.

Next we show $w^{nh}_\psi(T_n)\geq q$, i.e.,  $w^{nh}_\psi(e)\geq q$ for all $e=uv\in E(T_n)$ in the no-hole anti-$n$-labeling $\psi$.  If $e\in E(T_{n-m-1})$, then $w^{nh}_\psi(T_{n})\geq \frac{n}{2}$ since $w^{nh}_{\psi^{\prime}}(T_{n-m-1})\geq \frac{n}{2}-m$ by the induction hypothesis and $w^{nh}_\psi(e)\geq |\psi(u)-\psi(v)|=|\psi^{\prime}(u)-\psi^{\prime}(v)|+m$. If $e\notin E(T_{n-m-1})$, then $e$ is incident to $u$. Note that $u$ is labeled by $n$ and its neighboring vertices are labeled by some integer $i\leq \frac{n}{2}$. We have $w^{nh}_\psi(e)\geq \frac{n}{2}$. Hence,  $mc^{nh}_n(T_n)\geq q$(see Fig.2(3)).
\end{proof}

\begin{rem}
For an arbitrary  bipartition $(X_{1},X_{2})$, $|X_{1}|=q_{1}\leq |X_{2}|=q_{2}$, there is a tree $T_n$ such that $mc^{nh}_n(T_n)=q_{1}$.  We consider the tree $T_{n}$ with  bipartition $(X_{1},X_{2})$ as following: $T_{n}$  is obtained by joining $q_{1}-1$ new vertices to vertex of degree 1 in the star graph $K_{1,q_{2}}$. Since $\Delta(T_{n})=q_{2}$, then $mc^{nh}_n(T_n)\leq n-q_{2}=q_{1}$ by Observation \ref{obs 7}. Therefore, $mc^{nh}_n(T_n)=q_{1}$ by Theorem~\ref{thm 14}.
\end{rem}

We also pose a conjecture below.

\begin{con}\label{con 15}
 For a tree $T_n$ with bipartition $(X_{1},X_{2})$, $X_{i}=q_{i}$, $i=1,2$, we have $mc^{nh}_n(T_n)=q$, where $q=min\{q_{1},q_{2}\}$.
\end{con}

\input{g6.TpX}

\subsection{$mc^{nh}_{mn}(G)$ of 2-Dimensional grids $P_{m}\times P_{n}$}
In the Subsection, we generalize the result on paths to 2-Dimensional grids.
\begin{obs}\label{obs 16}
Let $G$ is a 2-Dimensional grid $P_{m}\times P_{n}$, then $mc^{nh}_{mn}(G)\geq\lfloor\frac{mn-m}{2}\rfloor$, where $m=\min\{m,n\}$.
\end{obs}
\begin{proof}
We look the $P_{m}\times P_{n}$ grid (i.e., $m$ rows and $n$ columns) as a chessboard. Like in the chessboard, we have white and black alternating squares (see Fig.3).

(i) If at least one of $m$ and $n$ is even (i.e., $mn$ is even), we have in the ``white'' squares the labels from the range $[1,\frac{mn}{2}]$ and in the ``black'' squares the labels from the range $[\frac{mn}{2}+1,mn]$. Without loss of generality, we assume that the left upper square is white. Take the following labeling $\psi$: put 1 in the left upper corner (put $\frac{mn}{2}+1$ in the second square in the first row of grid, resp.) and subsequently put in the white (black) squares from left to right and row by row the upper range labels: 2,3,$\ldots$,$\frac{mn}{2}$ ($\frac{mn}{2}+2,\frac{mn}{2}+3,\ldots,mn$, resp.).

Let $v$ be labelled by $i$, $i\leq\frac{mn}{2}$ ($i>\frac{mn}{2}$, resp.), then the vertices adjacent to $v$ are labelled by $i+\frac{mn}{2},i+\frac{mn}{2}-1,i+\lfloor\frac{mn-m}{2}\rfloor,i+\lfloor\frac{mn+m}{2}\rfloor$ ($i-\frac{mn}{2},i-\frac{mn}{2}+1,i-\lfloor\frac{mn-m}{2}\rfloor,i-\lfloor\frac{mn+m}{2}\rfloor$, resp.). Hence, $mc^{nh}_{mn}(G)\geq\lfloor\frac{mn-m}{2}\rfloor$ (see Fig.3(1)).

(ii) If $m$ and $n$ are odd (i.e., $mn$ is odd), we have in the ``white'' squares the labels from the range $[1,\frac{mn+1}{2}]$ and in the ``black'' squares the labels from the range $[\frac{mn+1}{2}+1,mn]$. Take the following labeling $\psi$: put 1 in the left upper corner (put $\frac{mn+1}{2}+1$ in the second square in the first row of grid, resp.) and subsequently put in the white (black, resp.) squares from left to right and row by row the upper range labels: 2,3,$\ldots$,$\frac{mn+1}{2}$ ($\frac{mn+1}{2}+2,\frac{mn+1}{2}+3,\ldots,mn$, resp). We have $mc^{nh}_{mn}(G)\geq\frac{mn-m}{2}$ by the argument of (i) (see Fig.3(2)).
\end{proof}

\input{g5.TpX}

\begin{con}\label{con 17}
Let $G$ is a 2-Dimensional grid $P_{m}\times P_{n}$, then $mc^{nh}_{mn}(G)=\lfloor\frac{mn-m}{2}\rfloor$, where $m=\min\{m,n\}$.
\end{con}
\subsection{$mc^{nh}_{2^{n}}(G)$ of $n$-cubes}

\begin{thm}\label{obs 18}
For a cycle $C_n$ of length $n$, $mc^{nh}_n(C_n)=\lfloor\frac{n-1}2\rfloor$.
\end{thm}
\begin{proof}
Since $\delta(C_n)=2$, $mc^{nh}_n(C_n)\leq \lfloor\frac{n-1}2\rfloor$ according to Observation ~\ref{obs 7}.

\input{g4.TpX}

Now we  show that $mc^{nh}_n(C_n)\geq \lfloor\frac{n-1}2\rfloor$. It suffices to show that there is a labeling $\psi$ such that $w^{nh}_{\psi}(C_n)=\lfloor\frac{n-1}2\rfloor$. Let $v_{1},v_{2},\ldots, v_{n}$ be the vertices of $C_{n}$ such that $v_{i}$ is adjacent to $v_{(i+1)\mod n}$, $1\leq i\leq n$.
 Consider the following labeling:

(i)If $n$ is even, then we define
\[\psi(v_{i})=\left\{\begin{array}{ll}
1&\text{$i$ is 1},\\
n&\text{$i$ is 3},\\
\frac{i-1}2&\text{$i$ is odd and $i\neq1,3$},\\
\frac{n}2-1+\frac{i}2&\text{$i$ is even}.
\end{array}\right.\]

(ii)If $n$ is odd, then we define
\[\psi(v_{i})=\left\{\begin{array}{ll}
\frac{n+1}2-\frac{i-1}2&\text{$i$ is odd},\\
n+1-\frac{i}2&\text{$i$ is even}.
\end{array}\right.\]

It is easy to show that $w^{nh}_{\psi}(e)$, $e\in E(C_{n})$, defined above is $\frac{n}2$ or $\frac{n}2-1$ for even $n$, and $\frac{n-1}2$ or $\frac{n+1}2$ for odd $n$. Hence $mc^{nh}_n(C_n)\geq \lfloor\frac{n-1}2\rfloor$.
\end{proof}

An $n$-cube can be defined inductively as follows. An $1$-cube is a $P_{2}$. An $n$-cube $Q_n$ may be constructed from the disjoint union of two $(n-1)$-cubes $Q_{n-1}$, by adding an edge from each vertex in one copy of $Q_{n-1}$ to the corresponding vertex in the other copy. The joining edges form a perfect matching.

\begin{thm}\label{thm 19}
Let $Q_n$ be an $n$-cube. Then, for all $n\geq2$, $mc^{nh}_{2^{n}}(Q_n)\geq 2^{n-2}$.
\end{thm}
\begin{proof}
\input{g2.TpX}
We show $mc^{nh}_{2^{n}}(Q_n)\geq 2^{n-2}$ by constructing a no-hole anti-$2^{n}$-labeling $\psi_{n}$ such that $w^{nh}_{\psi_{n}}(Q_n)\geq2^{n-2}$. If $n=2$, then $Q_n=C_{4}$. By Theorem \ref{obs 18},  $mc^{nh}_{2^{2}}(Q_2)=1\geq 2^{2-2}$. Let $\psi_{2}$ be the optimal no-hole anti-$2^{2}$-labeling  defined in Theorem \ref{obs 18} of  $Q_2$. Clearly, for each edge $e$ of $Q_2$,  one end of $e$ has label at most $2^{2-1}=2$ and the other end of $e$ has label greater $2$ under $\psi_{2}$, see Fig.5(1). For $m\leq n$, we assume there exists a labeling $\psi_m$ such that $w^{nh}_{\psi_{m}}(Q_m)\geq2^{m-2}$, and  one end has label at most $2^{m-1}$ and the other end has label greater $2^{m-1}$ for each edge in $Q_m$. We next construct the labeling $\psi_{n+1}$ satisfying the assumption above for $Q_{n+1}$ from the labeling $\psi_{n}$ defined above of $Q^{1}_n$ and $Q^{2}_n$ as follows.

 Note that an $n+1$-cube $Q_{n+1}$ can be obtained by adding a perfect matching $PM$ between two copies of an $n$-cube, denoted  by  $Q^{1}_n$ and $Q^{2}_n$ (Each edge of $PM$ joins two vertices having the same labels.).  We relabel the vertices with label $i>2^{n-1}$ in $Q^{1}_n$ by $i+2^{n-1}$, and we relabel the vertices with label $i\leq 2^{n-1}$ in $Q^{2}_n$ by $i+2^{n}+2^{n-1}$.

We next show that the assumption above holds for $\psi_{n+1}$ in $Q_{n+1}$. Let $e=uv$ be an edge of $E(Q_{n+1})$. We firstly assume $e\in E(Q^{1}_n)$ and   $\psi_{n}(u)>\psi_{n}(v)$. By the induction hypothesis, we have $\psi_{n}(u)>2^{n-1}$,  $\psi_{n}(v)\leq 2^{n-1}$ and  $\psi_{n}(u)-\psi_{n}(v)\geq2^{n-2}$. Therefore $\psi_{n+1}(u)=\psi_{n}(u)+2^{n-1}>2^{n}$,  $\psi_{n+1}(v)=\psi_{n}(v)\leq2^{n-1}<2^{n}$, and $w^{nh}_{\psi_{n+1}}(e)=|\psi_{n+1}(u)-\psi_{n+1}(v)|=\psi_{n+1}(u)-\psi_{n+1}(v)=\psi_{n}(u)+2^{n-1}-\psi_{n}(v)\geq2^{n-1}+2^{n-2}>2^{n-1}$ according to the definition of $\psi_{n+1}$. If $e\in E(Q^{2}_n)$ and we suppose $\psi_{n}(u)>\psi_{n}(v)$. Then $\psi_{n}(u)>2^{n-1}$,  $\psi_{n}(v)\leq 2^{n-1}$, and  $\psi_{n}(u)-\psi_{n}(v)<2^{n}$.  Therefore $\psi_{n+1}(u)=\psi_{n}(u)<2^{n}$,  $\psi_{n+1}(v)=\psi_{n}(v)+2^{n}+2^{n-1}>2^{n}$, and $w^{nh}_{\psi_{n+1}}(e)=|\psi_{n+1}(u)-\psi_{n+1}(v)|=\psi_{n+1}(v)-\psi_{n+1}(u)=\psi_{n}(v)+2^{n}+2^{n-1}-\psi_{n}(u)>2^{n-1}$. Finally, we assume $e\in E(PM)$. Without loss of generality, we assume $u\in V(Q^{1}_{n})$ and $v\in V(Q^{2}_{n})$. Then $\psi_{n}(u)=\psi_{n}(v)$. If $\psi_{n}(u)\leq 2^{n-1}$, then $\psi_{n+1}(u)=\psi_{n}(u)<2^{n}$,  $\psi_{n+1}(v)=\psi_{n}(v)+2^{n}+2^{n-1}>2^{n}$, and $w^{nh}_{\psi_{n+1}}(e)=2^{n}+2^{n-1}$.  If $\psi_{n}(u)>2^{n}$, then $\psi_{n+1}(u)=\psi_{n}(u)+2^{n-1}>2^{n}$, $\psi_{n+1}(v)=\psi_{n}(v)<2^{n}$, and $w^{nh}_{\psi_{n+1}}(e)=2^{n-1}$.   We complete the proof.
\end{proof}
\begin{thm}\label{thm 3}
Let $Q_3$ be a $3$-cube. Then $mc^{nh}_{8}(Q_3)=2$.
\end{thm}
\begin{proof}
We have $mc^{nh}_{8}(Q_3)\geq2$ by Theorem \ref{thm 19}.
We next show $mc^{nh}_{8}(Q_3)\leq2$ by contradiction. Suppose $mc^{nh}_{8}(Q_3)\geq3$. Let $\psi$ be an optimal labeling and we denote by $v_i$ the vertex with label $i$ under $\psi$. Then $v_{4}$ may only be  adjacent to vertices $v_{1},v_{7},v_{8}$, $v_{5}$ may only be  adjacent to vertices $v_{1},v_{2},v_{8}$, and $v_{6}$ may only be  adjacent to vertices $v_{1},v_{2},v_{3}$ in $Q_3$ due to $mc^{nh}_{8}(Q_3)\geq3$.  Note that $Q_3$ is a bipartite graph. Let the bipartition of $Q_3$ be $X,Y$, and $|X|=|Y|=2^{3-1}=4$. Without loss of generality, we assume $v_{4}\in X$. Then $v_{1},v_{7},v_{8}\in Y$, and  $v_{5},v_{6}\in X$. Hence, $v_{1},v_{2},v_{3},v_{7},v_{8}\in Y$, that is, $|Y|=5$, a contradiction.
\end{proof}

Note that the bound in Theorem \ref{thm 19} is sharp for $n=2,3$. We pose the following problem.

\begin{con}\label{con 21}
For all $n\geq2$, $mc^{nh}_{2^{n}}(Q_n)=2^{n-2}$.
\end{con}

\section{Anti-$L_{d}(2,1)$-labeling of graphs}

Given a simple graph $G=(V,E)$ and a positive number $d$, an $L_{d}(2,1)$-labeling of $G$ is a function $f: V(G)\rightarrow[0,\infty)$
such that whenever $x,y \in V$ are adjacent, if $|f(x)- f(Y)| \geq 2d$,
and whenever the distance between $x$ and $y$ is two, if $|f(x)- f(Y)| \geq d$. The $L_{d}(2,1)$-labeling number $\lambda(G, d)$ is the smallest number $m$ such that $G$ has an $L_{d}(2,1)$-labeling $f$ with $\max\{f(v):v\in V\} =m$. When $d=1$, the $L_{d}(2,1)$-labeling problem is the $L(2,1)$-labeling problem. The  $L(2,1)$-labeling problem of graphs has been discussed for many graph families, see~\cite{Calamoneri, Chang1, Chang2, Griggs, Yeh1, Yeh2}.

Similarly, we define the \emph{anti-$L_{d}(2,1)$-labeling} problem: Given a simple graph $G=(V,E)$ and a positive number $d$, a labeling of $G$ is a function $f: V(G)\rightarrow[1,k]$
such that $|f(x)- f(Y)| \geq 2d$ if $xy \in E(G)$, $|f(x)- f(Y)| \geq d$ if $d(x,y)=2$. The \emph{anti-$L_{d}(2,1)$-labeling number} of $G$, denoted by $mc^{\lambda}_{k}(G)$, is the largest number $2d$.

By the proofs of Observation \ref{obs 1} and Observation \ref{obs 3}, we have the results of Observation \ref{obs 22} and Observation \ref{obs 23} as following.

\begin{obs}\label{obs 22}
If  $H$ is a subgraph of $G$, then  $mc^{\lambda}_k(H)\geq mc^{\lambda}_k(G)$.
\end{obs}

\begin{obs}\label{obs 23}
If $G=G_{1}\cup G_{2}$, then $mc^{\lambda}_k(G)=\min\{mc^{\lambda}_k(G_{1}),mc^{\lambda}_k(G_{2})\}$.
\end{obs}

\begin{lem}[\cite{Griggs}]\label{lem 24}
$\lambda(G,d)=d\cdot\lambda(G,1)$ for a non-negative integer $d$.
\end{lem}
\begin{lem}[\cite{Gonalves}]\label{lem 25}
$\lambda(G,1)\leq \Delta^{2}+\Delta-2$.
\end{lem}
\begin{thm}\label{thm 26}
Let $G$ is a simple graph, then
$mc^{\lambda}_{k}(G)\geq2\lfloor\frac{k-1}{\Delta^{2}+\Delta-2}\rfloor$.
\end{thm}
\begin{proof}
Suppose $\lambda(G,d)=m$  for a graph $G$. Then there exists a labeling $f: V(G)\rightarrow[0,m]$ such that whenever $x,y \in V$ are adjacent, if $|f(x)- f(Y)| \geq 2d$,
and whenever the distance between $x$ and $y$ is two, if $|f(x)- f(Y)| \geq d$. Therefore, there exists a labeling $\psi: V(G)\rightarrow[1,m+1]$, such that $w_{\psi}(G)=2d$ for $k=m+1=\lambda(G,d)+1$. According to Lemma \ref{lem 24}, there exists a labeling $\psi$, such that $w^{\lambda}_{\psi}(G)=2\frac{k-1}{\lambda(G,1)}$ for all $k$.
Therefore $mc^{\lambda}_{k}(G)\geq2\lfloor\frac{k-1}{\lambda(G,1)}\rfloor$ for all $k$  according to the definition of anti-$L_{d}(2,1)$-labeling number $mc^{\lambda}_{k}(G)$. Combining  with Lemma \ref{lem 25}, we have
$mc^{\lambda}_{k}(G)\geq2\lfloor\frac{k-1}{\Delta^{2}+\Delta-2}\rfloor$.
\end{proof}
%$w^{\lambda}_{\psi}(G)=2d$ for $k\leq d(\Delta^{2}+\Delta-2)+1$, that is $w^{\lambda}_{\psi}(G)=2\lfloor\frac{k-1}{\Delta^{2}+\Delta-2}\rfloor$
\begin{thm}\label{lem 27}
If $mc^{\lambda}_{k}(G)=2d$ for a positive number $k$, then
$\frac{k-1}{d+1}<\lambda(G,1)\leq\frac{k-1}{d}$.
\end{thm}
\begin{proof}
Suppose $mc^{\lambda}_{k}(G)=2d$  for a graph $G$. Then $\lambda(G,d)+1\leq k <\lambda(G,d+1)+1$. In fact, it is obvious that  $\lambda(G,d)+1\leq k$, since  $G$ has an $L_{d}(2,1)$-labeling  for all  positive number $k$ and $\lambda(G, d)$ is the smallest number $m$ such that $G$ has an $L_{d}(2,1)$-labeling $f$. Suppose $k \geq\lambda(G,d+1)+1$. Then there exists a labeling $\psi$, such that $w^{\lambda}_{\psi}(G)=2(d+1)$. Hence $mc^{\lambda}_{k}(G)\geq 2(d+1)$ according to the definition of $mc^{\lambda}_{k}(G)$, a  contradiction. Hence, $d\cdot \lambda(G,1)+1\leq k <(d+1)\cdot\lambda(G,1)+1$ combining  with Lemma \ref{lem 24}, that is $\frac{k-1}{d+1}<\lambda(G,1)\leq\frac{k-1}{d}$.
\end{proof}

\section{Acknowledgements}
The research is supported by NSFC (No.11671296), Research Project Supported by Shanxi Scholarship Council of China,  Program for the Innovative Talents of Higher Learning Institutions of Shanxi (PIT).

\end{document}